\pdfoutput=1
\RequirePackage{ifpdf}
\ifpdf % We are running pdfTeX in pdf mode
\documentclass[pdftex]{sigma}
\else
\documentclass{sigma}
\fi

\numberwithin{equation}{section}

\newtheorem{Theorem}{Theorem}[section]
\newtheorem{Corollary}[Theorem]{Corollary}
\newtheorem{Lemma}[Theorem]{Lemma}
\newtheorem{Proposition}[Theorem]{Proposition}
 { \theoremstyle{definition}
\newtheorem{Definition}[Theorem]{Definition}
\newtheorem{Example}[Theorem]{Example}
\newtheorem{Remark}[Theorem]{Remark} }

\begin{document}

\allowdisplaybreaks

\newcommand{\arXivNumber}{1503.09023}

\renewcommand{\thefootnote}{$\star$}

\renewcommand{\PaperNumber}{092}

\FirstPageHeading

\ShortArticleName{Dif\/ferential Galois Theory and Lie Symmetries}

\ArticleName{Dif\/ferential Galois Theory and Lie Symmetries\footnote{This paper is a~contribution to the Special Issue on
Algebraic Methods in Dynamical Systems.
The full collection is available at
\href{http://www.emis.de/journals/SIGMA/AMDS2014.html}{http://www.emis.de/journals/SIGMA/AMDS2014.html}}}

\Author{David {BL\'AZQUEZ-SANZ}~$^\dag$, Juan J.~{MORALES-RUIZ}~$^\ddag$ and Jacques-Arthur {WEIL}~$^\S$}

\AuthorNameForHeading{B.~Bl\'azquez-Sanz, J.J.~Morales-Ruiz and J.-A.~Weil}

\Address{$^\dag$~Universidad Nacional de Colombia, Colombia}
\EmailD{\href{mailto:dblazquezs@unal.edu.co}{dblazquezs@unal.edu.co}}
\URLaddressD{\url{https://sites.google.com/a/unal.edu.co/dblazquezs/}}

\Address{$^\ddag$~Universidad Polit\'ecnica de Madrid, Spain}
\EmailD{\href{mailto:juan.morales-ruiz@upm.es}{juan.morales-ruiz@upm.es}}

\Address{$^\S$~Universit\'e de Limoges, France}
\EmailD{\href{mailto:jacques-arthur.weil@unilim.fr}{jacques-arthur.weil@unilim.fr}}

\ArticleDates{Received March 31, 2015, in f\/inal form November 11, 2015; Published online November 20, 2015}

\Abstract{We study the interplay between the dif\/ferential Galois group and
the Lie algebra of inf\/initesimal symmetries of systems of linear dif\/ferential
equations. We show that some symmetries can be seen as solutions of a hierarchy of linear
dif\/ferential systems. We show that the existence of  rational symmetries constrains the dif\/ferential Galois group in the system in a way that depends of the Maclaurin series of the symmetry along the zero solution.}

\Keywords{linear dif\/ferential system; Picard--Vessiot theory; dif\/ferential Galois theory; inf\/initesimal symmetries}

\Classification{12H05; 34M15; 34A26}

\renewcommand{\thefootnote}{\arabic{footnote}}
\setcounter{footnote}{0}

\section{Introduction}

Dif\/ferential Galois theory and Lie symmetries are two dif\/ferent theoretical
frameworks designed to deal with similar mathematical problems: the integration,
reduction, classif\/ication and listing of solutions of dif\/ferential equations.
Both theories appeared simultaneously at the end
of 19th century. However, the links between them remained hidden for a long time, mostly because
of the apparent walls that separate mathematical disciplines.
Dif\/ferential Galois theory appears to be central to dif\/ferential algebra. On the other hand,
the theory of Lie symmetries belongs to the realm
of local dif\/ferential geometry. For a general exposition of both theories
we refer the readers to \cite{CrespoHajto2011,PutSinger2003} and~\cite{KLR,Olver} respectively.
Throughout this paper, ``dif\/ferential Galois theory'' refers to the Galois theory of systems of linear dif\/ferential equations,
also called the Picard--Vessiot theory.

The works of Ziglin, Morales, Ramis, Sim\'o, Churchill and others show that
dif\/ferential Galois groups may measure obstructions to the existence of f\/irst integrals of
Hamiltonian systems; regarding this approach to non-integrability of dynamical systems, see the  recent survey~\cite{Morales2015} and references therein.
This suggests an interplay between Lie symmetries and Galois group.
Indeed, for Hamiltonian systems, the (musical) duality
induced by the symplectic structure transforms a f\/irst integral into a Lie symmetry; so,  obstructions to the existence of f\/irst integrals
induce  obstructions to the existence of  Hamiltonian vector f\/ields of Lie symmetries.

In~\cite{AZ}, Ayoul and Zung study a more general def\/inition of integrability (Bogoyavlensky integrability~\cite{Bog1996}), which includes directly Lie symmetries and they show again how the Galois groups of variational equations appear to give obstructions to the existence of such symmetries.

B. Malgrange has also suggested a link between Lie symmetries and his non-linear dif\/ferential Galois theory
(see, e.g.,~\cite[Remark~iii, p.~224]{Malgrange2002}).
The general idea is that  dif\/ferential Galois groups and Lie symmetries should be related in the sense that
more symmetries should imply a~smaller Galois group.
Our results elaborate on this idea to turn it into precise
statements which apply to  systems of linear dif\/ferential equations.

In order to link Lie symmetries with dif\/ferential Galois theory,  we follow a
geometrical approach  developed by the f\/irst two authors in~\cite{BM2010}.
Some general results about symmetries were already stated
in \cite[Section~6]{BM2012},
but in a more general context of automorphic systems. In particular, it is implicit in~\cite{BM2012}
that the {\it eigenring}~\cite{Barkatou2007,Singer1996} consists of vertical Lie symmetries.
Our approach  here is less abstract and more explicit (see Section~\ref{section6.2}). In
connection with that,  we point out that the relevance of the eigenring for the symmetries of linear dif\/ferential equations
 was discovered independently by C.~Jensen  in the nice paper~\cite{Jensen2005}, but without any mention of a relationship
 with Picard--Vessiot theory.

 The results contained in this work may be summarized as follows:
\begin{itemize}\itemsep=0pt
\item[(a)] The search for symmetries of systems of linear dif\/ferential
equations may be reduced to the search for a particular kind of symmetries, namely homogeneous
polynomial vertical symmetries (Lemma~\ref{Maclaurin}, Proposition~\ref{Maclaurin2}).
\item[(b)] We focus on polynomial vertical inf\/initesimal symmetries and show that these are
solutions of associated linear dif\/ferential equations that can be deduced from the given equation (via the Tannakian theory). The dif\/ferential Galois theories of those
equations is related to that of our original system (Theorem~\ref{t:sym_ci}).
\item[(c)] The Galois group determines the Lie algebra of polynomial vertical symmetries with coef\/f\/icients in the ground f\/ield (Theorem~\ref{t:Lie_Gal}).
\item[(d)] Each non-trivial vertical polynomial symmetry with coef\/f\/icients in the ground f\/ield places constrains on the Galois group. Thus, the bigger the symmetry algebra, the smaller the Galois group. In several cases, a single symmetry may force the group to be abelian or solvable (Theorems~\ref{t:decomposer} and~\ref{t:solver}, Corollaries~\ref{t:sym_lv} and~\ref{c:complete}).
\end{itemize}

Section~\ref{section2} contains the basic def\/initions of symmetries. Section~\ref{section3} studies polynomial sym\-met\-ries and establishes part~(a) above.
Section~\ref{section4} gives a dictionary between graded polynomial vector f\/ields, linear actions and corresponding linear dif\/ferential systems.
A geometric def\/inition of the dif\/ferential Galois group is given in Section~\ref{section5} and the comparisons with symmetries are derived in Section~\ref{section6}.

The problem of linking dif\/ferential Galois theory with Lie symmetries of dif\/ferential operators has
been studied by C.~Athorne in~\cite{Athorne1997} and by W.R.~Oudshoorn and M.~van der Put
in~\cite{Vanderput2002}.
This restriction to dif\/ferential operators seems to complicate the relations between the Galois group and the Lie symmetries.
So, in this work, we adopt a slightly dif\/ferent point of view. Instead of
considering higher-order linear dif\/ferential equations, we study
systems of f\/irst-order linear dif\/ferential equations. This leads us to a broader notion of inf\/initesimal
symmetry which has an explicit relation with the dif\/ferential Galois group.
In Appendix~\ref{Appendix_A}, we propose a  comparison between the def\/initions
of inf\/initesimal symmetries in the contexts of higher-order equations and f\/irst-order systems.

\section{Characteristic and vertical symmetries}\label{section2}

Let ${\mathcal U}$ be an open subset of the complex projective line
$\mathbb{CP}_1$. By a \emph{function field} $\mathcal K$
we mean a~subf\/ield of the f\/ield of meromorphic functions
on~${\mathcal U}$ such that $\mathcal K$ contains the constants~$\mathbb C$ and is
closed with respect the derivation~$\frac{d}{dx}$. Clearly,
f\/ields of rational functions, elliptic functions, etc.\ are
function f\/ields.
As shown by Seidenberg in \cite{Seidenberg1958,Seidenberg1969}, any dif\/ferential f\/ield which is f\/initely generated over
$\mathbb{Q}$ can be embedded into such a function f\/ield.
We f\/ix from now on the function
f\/ield $\mathcal K$ and its domain~${\mathcal U}$.

We consider a system of linear dif\/ferential equations with
coef\/f\/icients in the function f\/ield~$\mathcal K$
\begin{gather}\label{LE}
\frac{dy}{dx} = A(x)y, \qquad y = (y_1,\ldots,y_n), \qquad A(x)\in \mathfrak{gl}(n,\mathcal K).
\end{gather}
The poles of $A(x)$ are called the \emph{singularities}. The point at inf\/inity may be
considered as a~singularity depending on its behavior after a suitable change
of coordinates. The graphs of the solutions of~\eqref{LE} are the integral curves of the associated
vector f\/ield:
\begin{gather*}%\label{VF}
X = \frac{\partial}{\partial x} + \sum_{i,j=1}^n a_{ij}(x)y_j \frac{\partial}{\partial {y_i}},
\end{gather*}
which is a meromorphic vector f\/ield in ${\mathcal U}\times\mathbb C^n$.

An \emph{infinitesimal symmetry}
of $X$ is an analytic vector f\/ield $Y$ def\/ined in some
open subset of ${\mathcal U}\times\mathbb C^n$ such that
$[Y,X]$ and~$X$ are linearly dependent in their common
domain of def\/inition.
In particular, vector f\/ields of the form $\gamma X$ where
$\gamma$ is an analytic function def\/ined in some open subset
of ${\mathcal U}\times \mathbb C^n$ are inf\/initesimal symmetries of $X$.
They are called \emph{characteristic symmetries} of~$X$. Since the def\/inition of inf\/initesimal symmetries
is local, we have \emph{sheaves} of inf\/initesimal symmetries
and characteristic symmetries of~$X$ in ${\mathcal U}\times \mathbb C^n$.

The Lie bracket of two inf\/initesimal symmetries is also an inf\/initesimal
symmetry. Hence, inf\/initesimal symmetries form a Lie algebra sheaf.
Characteristic symmetries form an ideal of the Lie
algebra sheaf of inf\/initesimal symmetries.
An inf\/initesimal symmetry $Y$ is a
\emph{vertical symmetry} if it is tangent to the f\/ibers of the
canonical projection ${\mathcal U}\times \mathbb C^n\to {\mathcal U}$, that is
$Yx = 0$. Its expression
in coordinates takes the form
\begin{gather}\label{VS}
Y = \sum_{i=1}^n f_i(x,y)\frac{\partial}{\partial y_i}.
\end{gather}
If $Y$ is a vertical symmetry, then the Lie bracket $[X,Y]$ vanishes.

The Lie algebra sheaf of \emph{vertical symmetries}
is canonically isomorphic to the quotient Lie algebra sheaf
of all inf\/initesimal symmetries modulo the ideal
of characteristic symmetries.
If~$Y$ is an inf\/initesimal symmetry, we can take its
vertical representative: $\tilde Y = Y  - (Yx)X.$

By this reduction, the algebra of vertical symmetries is isomorphic to the algebra of
inf\/ini\-te\-si\-mal symmetries modulo the ideal of characteristic symmetries.
Thus, in order to study the sym\-met\-ries
of~\eqref{LE}, it suf\/f\/ices to consider vertical symmet\-ries.
 We consider symmetries that are def\/ined in open
subsets of the form $\mathcal V \times \mathbb C^n$
with $\mathcal V \subseteq \mathcal U$. Such symmetries can be seen as sections
of a sheaf def\/ined in~$\mathcal U$.

\begin{Definition}
The sheaf ${\bf sym}_X$ in ${\mathcal U}$ assigns to
each open subset ${\mathcal V}\subseteq {\mathcal U}$ the Lie algebra of vertical
inf\/initesimal symmetries of X def\/ined in
${\mathcal V}\times\mathbb C^n$
\begin{gather*}
{\bf sym}_X({\mathcal V}) = \big\{Y\in{\mathfrak X}_{\textrm{an}}({\mathcal V}\times\mathbb C^n) \, |\,
[Y,X]=0 \mbox{ and }  Yx = 0 \big\},
\end{gather*}
where ${\mathfrak X}_{\textrm{an}}({\mathcal V}\times\mathbb C^n)$ stands for the Lie algebra of
analytic vector f\/ields in ${\mathcal V}\times\mathbb C^n$.
\end{Definition}

Our objective is to describe the sections of this sheaf ${\bf sym}_X$
and its relation with the closed form solutions and the Picard--Vessiot
theory of the system~\eqref{LE}. From now on, when we mention a
\emph{symmetry of~$X$} we mean a section of ${\bf sym}_X$, that is, a
vertical inf\/initesimal symmetry.

{\sloppy Our def\/inition of inf\/initesimal symmetries is considered, for instance, in the reference book~\cite{KLR}, where the {\it vertical symmetries }
are called  {\it shuffling symmetries}. We prefer our terminology, because it has  a clearer geometrical meaning in our context of linear dif\/ferential equations (i.e., f\/iber bundles, although we will not use this geometrical terminology explicitly here). Actually, the relation $[X,Y]=0$  sometimes appears in the literature as being the direct def\/inition of an inf\/initesimal symmetry of~$X$, see
for instance~\cite{Bog1996,Bog2005}.

}

\section{Polynomial vertical  symmetries}\label{section3}

\subsection{Polynomial vertical  vector f\/ields}\label{section3.1}

\begin{Definition}
Let $Y$ be a vertical vector f\/ield def\/ined in ${\mathcal V}\times \mathbb C^{n}$ with
${\mathcal V}$ an open subset of ${\mathcal U}$
\begin{gather*}
Y = \sum_{i=1}^n f_i(x,y)\frac{\partial}{\partial y_i}.
\end{gather*}
We say that $Y$ is a \emph{polynomial vertical} vector f\/ield
when the components $f_i(x,y)$ are polynomials in the variables $y_1,\ldots, y_n$.
\end{Definition}

\begin{Example}\label{ex:1}
The (Euler) homogeneous vector f\/ield
\begin{gather*}
\vec h = \sum_{i=1}^n y_i\frac{\partial}{\partial y_i}
\end{gather*}
is polynomial vertical, indeed linear vertical.
It is a symmetry of any system of linear dif\/ferential equations.
Hence, it is a global section of ${\bf sym}_X$.
\end{Example}

The def\/initions of \emph{degree} and \emph{homogeneous components} of a
polynomial vertical  vector f\/ield are clear. Given a function f\/ield
$\mathcal K$ of meromorphic functions on ${\mathcal V}\subseteq {\mathcal U}$, we can also
speak of the \emph{polynomial vertical  vector fields with coefficients
in~$\mathcal K$}. They are the polynomial vertical  vector f\/ields
\begin{gather*}
Y = \sum_{i=1}^n P_i(x,y)\frac{\partial}{\partial y_i},
\end{gather*}
where $P_1(x,y),\ldots,P_n(x,y)$ are in $\mathcal K[y_1,\ldots,y_n]$.

\subsection{Homogeneous components of symmetries}\label{section3.2}

Let us consider a vertical inf\/initesimal symmetry $Y\in {\bf sym}_X({\mathcal V})$ with
${\mathcal V}\subseteq {\mathcal U}$.
We can develop the Maclaurin series for the components of $Y$ with respect the
variables $y_1,\ldots,y_n$, obtaining
\begin{gather*}
Y = \sum_{i=1}^n \sum_{\alpha\in \mathbb Z_+^n} g_{i,\alpha}(x)y^\alpha \frac{\partial}{\partial y_i},
\end{gather*}
where the functions $g_{i,\alpha}(x)$ are analytic functions on ${\mathcal V}$.
We decompose~$Y$ as a sum of its homogeneous components
\begin{gather*}
Y = Y_0 + Y_1 + Y_2 + \cdots,
\end{gather*}
where each $Y_j$ is a homogeneous polynomial vertical  vector f\/ield of degree~$j$ (i.e., with respect to the~$y$ variables)
in ${\mathcal V}\times \mathbb C^n$.

\begin{Lemma}\label{Maclaurin}
Let $Y\in {\bf sym}_X({\mathcal V})$ with ${\mathcal V}\subseteq {\mathcal U}$ be a~symmetry of~$X$.
All the  homogeneous com\-po\-nents $Y_j$ of its Maclaurin series are symmetries of~$X$.
\end{Lemma}

\begin{proof}
The Lie bracket can be computed componentwise because the map $[X,\bullet]$ is homogeneous of degree $0$ in its action
on vector f\/ields, so we have
\begin{gather*} %\label{graded-bracket}
0 = [Y,X] = [Y_0,X] + [Y_1, X] + [Y_2,X] + \cdots.
\end{gather*}
For each $j\geq 0$, $[X,Y_j]$ is a homogeneous polynomial vertical  vector f\/ield  of degree~$j$.
Thus, all the terms of the above series vanish and we have proved the result.
\end{proof}

\begin{Remark}\label{rMaclaurin}
The Maclaurin series of $Y$ depends only on the value of $Y$ in small neighbourhood of ${\mathcal V}\times \{0\}$ in ${\mathcal V}\times \mathbb C^n$.
Assume that $Y$ is a rational vertical symmetry of $X$,
\begin{gather*}
Y = \sum_{j=1}^n f_i(x,y)\frac{\partial}{\partial y_i},
\qquad f_i(x,y) \in \mathbb C(x,y_1,\ldots,y_n),
\end{gather*}
whose polar set does not contain the curve
$\mathbb{CP}_1\times \{0\}$ in $\mathbb{CP}_1\times \mathbb C^n$
(i.e., the denominators do not vanish indentically for $y=\vec{0}$).
Then it admits a Maclaurin expansion in $y$ and
Lemma~\ref{Maclaurin} shows that each homogeneous component
$Y_0$, $Y_1$, $Y_2$, etc.\ of $Y$ is a (homogeneous) polynomial vertical  symmetry with coef\/f\/icients in
$\mathbb C(x)$.
\end{Remark}

\subsection{Homogeneous polynomial vertical  symmetries}\label{section3.3}

The sheaf ${\bf sym}_X$ contains the subsheaf of \emph{polynomial vertical symmetries} that we denote by ${\bf sym}^{<\infty}_X$.
Lemma \ref{Maclaurin} implies that the homogeneous components
of polynomial vertical  sym\-met\-ries are also sym\-met\-ries. Hence, we have
a decomposition
\begin{gather*}
{\bf sym}^{<\infty}_X = \bigoplus_{n=0}^\infty {\bf sym}^{r}_X,
\end{gather*}
where ${\bf sym}^{r}_X$ stands for the sheaf of homogeneous polynomial vertical
symmetries of~$X$ of deg\-ree~$r$. These objects can be interpreted
simultaneously in two complementary ways, as sheaves and as dif\/ferential
varieties:
\begin{itemize}\itemsep=0pt
\item[(a)] As a sheaf, ${\bf sym}^{r}_X$ maps each open subset ${\mathcal V}\subseteq {\mathcal U}$ to
the set ${\bf sym}^{r}_X({\mathcal V})$ of homogeneous polynomial vertical  symmetries of $X$
def\/ined in ${\mathcal V}\times\mathbb C^n$.
\item[(b)] As a dif\/ferential variety, ${\bf sym}^{r}_X$ maps each dif\/ferential f\/ield
extension $\mathcal K\subseteq\mathcal F$ to the set ${\bf sym}^{r}_X(\mathcal F)$
of homogeneous polynomial vertical  symmetries of $X$ with coef\/f\/icients in~$\mathcal F$.
Since the Lie bracket is computed algebraically, this makes perfect sense even if $\mathcal F$ is not a~function f\/ield. Our forthcoming Theorem~\ref{t:sym_lv}, stated and proved in Section~\ref{section6} below, tells that ${\bf sym}^{r}_X$
is in fact a linear dif\/ferential variety def\/ined over~$\mathcal K$.
\end{itemize}

If the function f\/ield $\mathcal K$ contains the rational functions,
then rational symmetries, as vector f\/ields in $\mathbb{CP}_1\times\mathbb C^n$,
can be always reduced to polynomial vertical  symmetries with coef\/f\/icients in~$\mathcal K$:

\begin{Proposition}\label{Maclaurin2}
Assume that $\mathcal K$ contains the field of rational functions~$\mathbb C(x)$.
Let $Y$ be a~rational vector field in $\mathbb{CP}_1\times\mathbb C^n$ which is
a non-characteristic rational symmetry of
$X$ and whose polar subset does not include the curve ${\mathcal U}\times \{0\}$ in ${\mathcal U}\times \mathbb C^n$.
We consider the Maclaurin series
\begin{gather*}
Y - (Yx)X = Y_0 + Y_1 + Y_2 + \cdots,
\end{gather*}
where each $Y_r$ is a homogeneous polynomial vertical  vector field of degree~$r$.
Then, for each $r$, $Y_r\in {\bf sym}^r_X(\mathcal K)$ and for at least one index~$r$,
$Y_r$ is not zero.
\end{Proposition}

\begin{proof}
Lemma \ref{Maclaurin} and its Remark \ref{rMaclaurin} show that
the vector f\/ields $Y_r$ are symmetries. We only need  to check that they have coef\/f\/icients
in $\mathcal K$. Let us consider the expression of $Y$ in coordinates
\begin{gather*}
Y = h(x,y)\frac{\partial}{\partial x} + \sum_{j=1}^n f_j(x,y)\frac{\partial}{\partial y_j}.
\end{gather*}
Then
\begin{gather*}
Y - (Yx)X = \sum_{j=1}^n H_j(x,y)\frac{\partial}{\partial y_j},\qquad
H_j(x,y) = \sum_{i=1}^n (f_j(x,y) - h(x,y)a_{ji}(x)y_i),
\end{gather*} and
\begin{gather*}
Y_r = \sum_{j=1}^n\sum_{|\alpha| = r }
\frac{\partial^{|\alpha|} H_j}{\partial y^\alpha}(x,0)\frac{y^{\alpha}}{\alpha!}
\frac{\partial}{\partial y_j}.
\end{gather*}
A direct examination of the expressions shows that they have coef\/f\/icients in $\mathcal K$.
\end{proof}

Since the Lie bracket is a graded operation, the sheaves ${\bf sym}_X^r$ are not in general
Lie algebra sheaves. We have
\begin{gather*}
[\,\,,\,]\colon \ {\bf sym}_X^r \times {\bf sym}_X^s \to {\bf sym}_X^{r+s-1}.
\end{gather*}
For $n>1$ only ${\bf sym}_X^0$, ${\bf sym}_X^1$ and ${\bf sym}_X^0\oplus {\bf sym}_X^1$
are Lie algebra sheaves. Our next objective is to show that the sections of ${\bf sym}_X^r$
for each $r$ are solutions of a hierarchy of  linear dif\/ferential systems
canonically attached to~\eqref{LE}.

\section[Polynomial vector f\/ields in $\mathbb C^n$]{Polynomial vector f\/ields in $\boldsymbol{\mathbb C^n}$}\label{section4}

\subsection{The Lie algebra of polynomial vector f\/ields}\label{pol_vf}

Polynomial vertical  symmetries are polynomial vector f\/ields
along the f\/ibers of the projection ${\mathcal U} \times \mathbb C^n
\to \mathcal U$. In this section we will give some remarks about the structure
of the Lie algebra~$\mathfrak X[\mathbb C^n]^{<\infty}$ of
\emph{polynomial vector fields} in~$\mathbb C^n$
\begin{gather*}
\mathfrak X[\mathbb C^n]^{<\infty} =
\left\{\sum_{i=1}^n P_i(y)\frac{\partial}{\partial y_i} \colon  P_i(y)\in \mathbb C[y_1,\ldots,y_n]\right\}.
\end{gather*}
By taking homogeneous components, we have
 $\mathfrak X[\mathbb C^n] = \bigoplus_{r=0}^\infty \mathfrak X^r[\mathbb C^n]$,
where
\begin{gather*}
\mathfrak X^r[\mathbb C^n] =
\left\{\sum_{i=1}^n P_i(y)\frac{\partial}{\partial y_i}\colon P_i(y)\, \mbox{homogeneous of degree} \ r\right\}.
\end{gather*}
The Lie bracket respects the degree in the following way
\begin{gather*}
[\,,\,]\colon \ \mathfrak X^r[\mathbb C^n]\times
\mathfrak X^s[\mathbb C^n]
\to \mathfrak X^{r+s-1}[\mathbb C^n].
\end{gather*}

\begin{Remark}
For $n > 2$, exactly two of the homogeneous components of
$\mathfrak X[\mathbb C^n]$ are Lie subalgebras:
\begin{itemize}\itemsep=0pt
\item[(a)] The homogeneous component of degree zero $\mathfrak X^0[\mathbb C^n]$. It is the Lie algebra of the inf\/i\-ni\-te\-si\-mal generators of the action of the group of translations in
$\mathbb C^n$. It is an abelian Lie algebra canonically
isomorphic to~$\mathbb C^n$
\begin{gather*}
\mathfrak X^0[\mathbb C^n]\simeq \mathbb C^n, \qquad \frac{\partial}{\partial y_j}\mapsto e_j,
\end{gather*}
where  $\{e_1,\ldots,e_n\}$ stands for the canonical basis of $\mathbb C^n$.
\item[(b)] The homogeneous component of degree one, $\mathfrak X^1[\mathbb C^n]$. It is the Lie algebra of linear vector f\/ields
in~$\mathbb C^n$. It consists of the inf\/initesimal generators of
the action of the group of linear transformations of~$\mathbb C^n$.
It is canonically isomorphic to~$\mathfrak{gl}(n,\mathbb C)$
in the following sense:\vspace{-2mm}
\begin{gather*}
\mathfrak X^1[\mathbb C^n]\simeq \mathfrak{gl}(n,\mathbb C),
\qquad \sum_{i,j=1}^n a_{ij}y_j\frac{\partial}{\partial y_i}\mapsto A,
\end{gather*}
where $A$ stands for $n\times n$ matrix of entries $a_{ij}$.

Given an endomorphism $A\in\mathfrak{gl}(n,\mathbb C)$, we let
$\vec v_A$ denote  its corresponding linear vector f\/ield
in $\mathbb C^n$\vspace{-2mm}
\begin{gather*}
\vec v_A := \sum_{i,j=1}^n a_{ij}y_j\frac{\partial}{\partial y_i}
\quad \textrm{with the identity} \quad
\vec v_A(y) = \left.\frac{d}{d\varepsilon}\right|_{\varepsilon = 0} e^{\varepsilon A}y.
\end{gather*}
We may easily check that this morphism is in fact an anti-isomorphism of Lie algebras: for
any pair $(A,B)$ of matrices, we have
\begin{gather*}
[\vec v_A, \vec v_B] + \vec v_{[A,B]} = 0.
\end{gather*}
\end{itemize}
\end{Remark}

\subsection{Induced linear actions} \label{InducedLinearActions}

Let us consider the canonical action of ${\rm GL}(n,\mathbb C)$
on $\mathbb C^n$ by linear transformations
\begin{gather*}
{\rm GL}(n,\mathbb C) \times \mathbb C^n \to \mathbb C^n,
\qquad (A,y)\mapsto Ay,
\end{gather*}
By abuse of notation, we denote non degenerate matrices and their associated linear
transformation on $\mathbb C^n$ by the same symbols.
Let $A$ be a linear transformation and $Y$ be a homogeneous
polynomial vector f\/ield of degree $r$.

Viewing $A$ as a (linear) dif\/feomorphism of ${\mathbb C}^n$,
we let $A_*(Y)$ denote the vector f\/ield transformed by $A$. In general, for any dif\/feomorphism $F$,
we let $F_*(Y)(F(p)) = dF(Y(p))$,
i.e.,
\begin{gather*}
F_*(Y)(p) = dF\big(Y\big(F^{-1}(p)\big)\big).
\end{gather*}
This def\/ines a natural action of dif\/feomorphisms of ${\mathbb C}^n$  on vector f\/ields of~${\mathbb C}^n$.

It is easy to check that
$A_*(Y)$ is also a  homogeneous polynomial vector f\/ield of degree~$r$.
Thus, for each $r\geq 0$ we have an induced representation
\begin{gather*}
\Phi_r\colon \ {\rm GL}(n,\mathbb C) \to
{\rm GL}\big(\mathfrak X^r\big[\mathbb C^n\big]\big),\qquad \Phi_r(A)(Y) = A_*(Y),
\end{gather*}
which yields a linear representations of ${\rm GL}(n,\mathbb C)$
on the f\/inite-dimensional vector spa\-ces~$\mathfrak X^r[\mathbb C^n]$.
This action can be dif\/ferentiated at the identity
obtaining the inf\/initesimal action
\begin{gather*}\Phi'_r\colon \ \mathfrak{gl}(n,\mathbb C) \to
{\rm End}\big(\mathfrak X^r\big[\mathbb C^n\big]\big),\qquad
\Phi'_r(A)(Y) = \left.\frac{d}{d\varepsilon}\right|_{\varepsilon = 0}
\Phi\big(e^{\varepsilon A}\big)(Y).
\end{gather*}

The following remark is  key to connect the def\/inition of symmetry
with the dif\/ferential Galois theoretic aspects
of equation~(\ref{LE}).

\begin{Lemma}\label{keylemma}
The infinitesimal action $\Phi'$
of $\mathfrak{gl}(n,\mathbb C)$
in $\mathfrak X^r[\mathbb C^n]$
coincides up to a change of sign, by the canonical isomorphism between
$\mathfrak X^1[\mathbb C^n]$ and $\mathfrak{gl}(n,\mathbb C)$,
with the Lie bracket action of linear vector fields on~$\mathfrak X^r[\mathbb C^n]$. That is, for any endomorphism~$A$
and homogeneous polynomial vector field~$Y$  in~$\mathbb C^n$
\begin{gather*}
\Phi'_r(A)(Y) = - [\vec v_A, Y].
\end{gather*}
\end{Lemma}

\begin{proof}
Let us def\/ine, for each $A$ and $\varepsilon$,
$\sigma_\varepsilon\colon \mathbb C^n\to \mathbb C^n$, the map that sends
each $y\in\mathbb C^n$ to $e^{\varepsilon A}y$. Thus,
$\{\sigma_\varepsilon\}_{\varepsilon\in\mathbb C}$ is the f\/low of the vector
f\/ield $\vec v_A$. We have
\begin{gather*}\Phi'_r(A)(Y) = \left.\frac{d}{d\varepsilon}\right|_{\varepsilon = 0}
\Phi\big(e^{\varepsilon A}\big)(Y) =
\left.\frac{d}{d\varepsilon}\right|_{\varepsilon = 0} \sigma_{\varepsilon*}(Y)
= -{\rm Lie}_{\vec v_A} Y = -[\vec v_A,Y],\end{gather*}
by the usual geometric def\/inition of Lie derivative.
\end{proof}

\newpage

\section{Lie--Vessiot hierarchy and Galois group}\label{section5}

\subsection{The Lie--Vessiot hierarchy}\label{section5.1}

Here we recall some of the def\/initions from \cite{BM2012}, adapted to the particular case of linear equations.
Let $E$ be a f\/inite-dimensional complex vector space and $(E,\Psi)$ a linear representation
of~${\rm GL}(n,\mathbb C)$ in~$E$. The group morphism $\Psi$ induces
a Lie algebra morphism $\Psi'$:
\begin{gather*}\Psi\colon \  {\rm GL}(n,\mathbb C) \mapsto {\rm GL}(E),\qquad
\Psi'\colon \ \mathfrak{gl}(n,\mathbb C) \to {\rm End}(E).\end{gather*}
This morphism transports our linear dif\/ferential system \eqref{LE} to
a linear dif\/ferential system in ${\mathcal U} \times E$ with coef\/f\/icients in~$\mathcal K$:
\begin{gather}\label{induced}
\frac{dv}{dx} = \Psi'(A(x))(v).
\end{gather}
We say that  system (\ref{induced}) is the \emph{Lie--Vessiot system} induced by \eqref{LE}
in the representation $(E,\Psi)$.  They  are the geometric analog of the dif\/ferential systems obtained by Tannakian correspondence on tensor constructions in standard dif\/ferential Galois theory, see \cite{Katz,PutSinger2003}.

A  solution of the Lie--Vessiot system \eqref{induced} in ${\mathcal V}\subset {\mathcal U}$ is
an analytic map ${\mathcal V}\to E$ that satisf\/ies the equations.
Given a dif\/ferential
f\/ield extension $\mathcal K \subseteq \mathcal F$, a solution of the Lie--Vessiot system~\eqref{induced} in~$\mathcal F$
can be thought of as being an element of $E\otimes_{\mathbb C}\mathcal F$. If we take
a basis $v_1,\ldots,v_n$ of $E$ and denote by $\lambda_1,\ldots,\lambda_n$ their
corresponding linear coordinate functions, the functions $b_{ij}(x) = \lambda_i(\Psi'(A(x))v_j)$
are elements of~$\mathcal K$ and the dif\/ferential equation can be written
in coordinates
\begin{gather}\label{LE2}
\frac{d\lambda}{dx} = B(x)\lambda, \qquad \lambda = (\lambda_1,\ldots,\lambda_m), \qquad B(x)\in \mathfrak{gl}(m,\mathcal K).
\end{gather}
Note that, as the Lie--Vessiot construction is a Lie algebra morphism, the poles of $B(x)$ are exactly the poles of $A(x)$.

There is a natural relation between  solutions of~\eqref{LE} and of its induced Lie--Vessiot system~\eqref{induced}:
if $M(x)$ is a fundamental matrix of solutions of~\eqref{LE} def\/ined
in ${\mathcal V}\subseteq {\mathcal U}$ then, for all $v_0\in E$, $v(x) = \Psi(M(x))v_0$ is a particular
solution of the Lie--Vessiot system~\eqref{induced}.

\subsection{The Galois group}

The Lie--Vessiot systems induced by \eqref{LE} form a
hierarchy of dif\/ferential equations which encode the dif\/ferential algebraic properties
of \eqref{LE}. It allows us to give a ``geometric'' def\/inition of the dif\/ferential
Galois group. In this def\/inition we are concerned with two kind of solutions of
the Lie--Vessiot systems:
\begin{itemize}\itemsep=0pt
\item[(a)] We say that a solution $v(x)$ of the Lie--Vessiot system \eqref{induced} is
\emph{$\mathcal K$-rational} if it belongs to $E\otimes_{\mathbb C}\mathcal K$. This means
that $v(x)$ has its coordinates \eqref{LE2} in $\mathcal K$.
\item[(b)] We say that an element $w(x)$ of $E\otimes_{\mathbb C}\mathcal K$ is a
\emph{$\mathcal K$-exponential
pre-solution} of the Lie--Vessiot system \eqref{induced} if there is a function $b(x)\in \mathcal K$ such that
\begin{gather*}
\frac{dw}{dx} - \Psi'(A(x))(w) = -b(x)w.
\end{gather*}
It models the case in which $v(x) = \exp(\int b(x)dx)w(x)$ is a solution of the Lie--Vessiot system \eqref{induced}
or, equivalently, the class $\langle v(x)\rangle$ is a rational solution of
the equation in the projective space $\mathbb P(E)$ obtained by reducing the Lie--Vessiot system \eqref{induced}
by the Euler homogeneous vector f\/ield (see Example \ref{ex:1}) of symmetries.
\end{itemize}

Note that the concept of $\mathcal K$-exponential pre-solution extends that of
$\mathcal K$-rational solution: any $\mathcal K$-rational solution is a
$\mathcal K$-exponential pre-solution in which the multiplier $b(x)$ vanishes.

\begin{Definition} Let us f\/ix an $x_0\in\mathcal U$ which is not a singularity of \eqref{LE}.
We say that a non-degenerate matrix $\sigma\in{\rm GL}(n,\mathbb C)$
is Galoisian at $x_0$ if for any
linear representation $(E,\Psi)$ it satisf\/ies the two following conditions:
\begin{itemize}\itemsep=0pt
\item[(a)] For any $\mathcal K$-rational solution $v(x)$ of
any induced Lie--Vessiot system \eqref{induced}, $\Psi(\sigma)(v(x_0)) = v(x_0)$.
Note that,
if $v(x)$ is a $\mathcal K$-rational solution
then $v(x_0)$ is well def\/ined as an element of~$E$.
\item[(b)] For any $\mathcal K$-exponential pre-solution $w(x)$
of any induced Lie--Vessiot system~\eqref{induced}, for which~$w(x_0)$ is well def\/ined,
$w(x_0)$ is an eigenvector of~$\Psi(\sigma)$.
\end{itemize}
\end{Definition}

The Galoisian matrices at $x_0$ form a
group ${\rm Gal}(x_0,X)$, called the \emph{Galois group} of~\eqref{LE}
at the point $x_0$. It is the stabilizer of all the values at $x_0$ of $\mathcal K$-rational
solutions, and the lines spanned by the values at $x_0$ of $\mathcal K$-exponential pre-solutions
of induced Lie--Vessiot systems.

Although this geometric def\/inition may seem dif\/ferent from the standard ones from Picard--Vessiot theory, it produces the same group. Choose a normalized local solution matrix at $x_0$, i.e., one with initial condition being the identity at the point $x_0$; then
$v(x_0)$ will be the coordinates of the invariant~$v(x)$ on this normalized basis of solutions. Our def\/inition hence says that a~mat\-rix~$\sigma$ is in~${\rm Gal}(x_0,X)$ if and only if it admits all the (semi-)invariants of the (Picard--Vessiot) dif\/ferential Galois group as (semi-)invariants.

The following facts are well known in
dif\/ferential Galois theory (we refer the interested reader to \cite{CrespoHajto2011, PutSinger2003} for
a general exposition, or to \cite{BM2010, BM2012} for an exposition which is
consistent with our geometric def\/inition):
\begin{itemize}\itemsep=0pt
\item[(a)] The Galois group ${\rm Gal}(x_0,X)$ is an algebraic subgroup of
${\rm GL}(n,\mathbb C)$.
\item[(b)] The Galois groups at two dif\/ferent non-singular points $x_0$ and
$x_1$ are conjugated. We will write ${\rm Gal}(X)$ to denote this abstract
Galois group, that we call the Galois group of the equation \eqref{LE} over $\mathcal K$.
\item[(c)] The system \eqref{LE} is integrable by Liouvillian functions if and only if
${\rm Gal}(x_0,X)$ is a virtually solvable group, i.e., its identity component is solvable.
\end{itemize}

The next lemma encodes the expected Galois correspondence.
The reader may check
\cite[Proposition~5.4]{BM2010} for a geometric proof
that relies on Lie's reduction method and Chevalley theorem.

\begin{Lemma}\label{invariant}
Let $(E,\Psi)$ be a linear representation of ${\rm GL}(n,\mathbb C)$ and
let $z(x)$ be a solution of its corresponding induced Lie--Vessiot
system~\eqref{induced}. Then, $z(x)$ is a $\mathcal K$-rational
solution if and only if for all Galoisian matrices at $x_0$, $\sigma\in{\rm Gal}(x_0,X)$,
we have $\Psi(\sigma)z(x_0) = z(x_0)$.
\end{Lemma}

\section{Symmetries vs Galois}\label{section6}

\subsection{Symmetries as solutions of the Lie--Vessiot hierarchy}\label{section6.1}

A key of the relation between the Galois group and the symmetries
is the fact that polynomial vertical  vector f\/ields in ${\mathcal V}\subseteq {\mathcal U}$ of can be seen as
maps ${\mathcal V}\to \mathfrak X[\mathbb C^n]$. For a given polynomial vertical
vector f\/ield $Y$ and $x_0\in {\mathcal V}$ we will write~$Y(x_0)$ for \emph{the value of~$Y$ at~$x_0$}. It is a~polynomial vector f\/ield in $\mathbb C^n$ that corresponds
to the restriction of~$Y$ to the f\/ibre $\{x_0\}\times\mathbb C^n$. It is clear that,
for general vertical vector f\/ields, $[Y,Z](x_0) = [Y(x_0),Z(x_0)]$.
If we restrict our considerations
to homogeneous polynomial vertical  vector f\/ields of a f\/ixed degree $r$,  then
$\mathfrak X^r[\mathbb C^n]$ turns out to be a~f\/inite-dimensional complex
space. This will allow us to describe polynomial vertical  symmetries
as solutions of some systems of the Lie--Vessiot hierarchy.

\begin{Theorem}\label{t:sym_lv}
Let $Y$ be a homogeneous polynomial vertical  vector field of degree $r$
defined in ${\mathcal V}\times\mathbb C^n$ with ${\mathcal V}\subseteq {\mathcal U}$. Then, $Y$ is a symmetry of~\eqref{LE} if and only if, as a map from~${\mathcal V}$ to~$\mathfrak X^r[\mathbb C^n]$,
it is a solution of the Lie--Vessiot system induced in
the representation $(\Phi_r, \mathfrak X^r[\mathbb C^n])$ from Section~{\rm \ref{InducedLinearActions}}.

In other words,
$Y\in {\bf sym}_X({\mathcal V})$ if and only if it satisfies
\begin{gather}\label{induced2}
\frac{d Y}{dx} = \Phi_r'(A(x))Y.
\end{gather}
\end{Theorem}

\begin{proof} Recall that if $\vec v_A=\vec v_A(x)$ is the linear vertical  vector f\/ield in ${\mathcal U}\times\mathbb C^n$
corresponding to the matrix $A(x)$, then $X = \frac{\partial}{\partial x} + \vec v_A(x)$.
Let us compute the Lie bracket
\begin{gather*}
[X,Y] = \left[\frac{\partial}{\partial x} + \vec v_A(x), Y \right] = \frac{dY}{dx} +
[\vec v_A(x),Y].
\end{gather*}
Thus, $Y$ is a symmetry if and only if $\frac{dY}{dx} = - [\vec v_A(x), Y]$. Finally,
by Lemma~\ref{keylemma} we have that~$Y$ satisf\/ies the stated dif\/ferential equation if and
only if it is a symmetry.
\end{proof}

If we denote by $\mathfrak X[\mathbb C^n]^{<\infty}$ the polynomial vector f\/ields in $\mathbb C^n$, then:

\begin{Corollary}\label{t:sym_ci}
Let $x_0$ be a non-singular point, and ${\mathcal V}$ a simply-connected neighbourhood of~$x_0$
in~${\mathcal U}$.
Then, for each polynomial vector field $Y^{(0)}\in \mathfrak X[\mathbb C^n]^{<\infty}$ there is
a unique
polynomial ver\-tical  symmetry $Y\in {\bf sym}_X({\mathcal V})^{<\infty}$ such that $Y(x_0) = Y^{(0)}$. Moreover,
${\bf sym}_X({\mathcal V})^{<\infty}$ and $\mathfrak X[\mathbb C^n]^{<\infty}$ are isomorphic Lie algebras.
\end{Corollary}

\begin{proof}
The map ${\bf sym}_X({\mathcal V})^{<\infty}\to \mathfrak X[\mathbb C^n]^{<\infty}$, $Y\mapsto Y(x_0)$,
is a Lie algebra morphism
since the computation of the Lie bracket and the restriction to the f\/iber
$\{x_0\} \times \mathbb C^n$ are commuting processes. We have to see
that it is an isomorphism. Let~$r$ be
the degree of $Y^{(0)}$ and
\begin{gather*}
Y^{(0)}= Y_0^{(0)} + Y_1^{(0)} + \cdots + Y_r^{(0)}
\end{gather*}
the decomposition of $Y^{(0)}$ in homogeneous components. Let $Y_k$ be
the solution in ${\mathcal V}$ of the Cauchy problem
\begin{gather*}
\frac{dY_k}{dx} = \Phi'_j(A(x))Y_k, \qquad Y_k(x_0) = Y^{(0)}_k.
\end{gather*}
The existence and uniqueness of the solution guarantees that
\begin{gather*}
Y  = Y_0 +  Y_1 + \cdots + Y_r
\end{gather*}
is the only polynomial vertical  symmetry such that $Y(x_0) = Y^{(0)}$.
\end{proof}

We will now write the system $(\ref{induced2})$ in more explicit form using tensor products.
\begin{Proposition}
Let $N:= {n+m-1 \choose m}$ denote the number of monomials of degree~$m$ in~$n$ va\-riables.
We define the matrix $ {\mathcal{A}_m}:=A \otimes \operatorname{Id}_N + \operatorname{Id}_n \otimes {\bf sym}^m(A^\star) $ of size $nN$.

The system $y'=  {\mathcal{A}_m} y$ has a rational solution
\begin{gather*}
y=\big[a_{1,1}(x),\ldots, a_{N,1}(x),\ldots,a_{1,j}(x),\ldots,a_{N,j}(x),\ldots,\\
\hphantom{y=\big[}{}a_{1,n}(x),\ldots,a_{N,n}(x)\big]^T,
 \qquad a_{i,j}(x) \in \mathcal{K}
\end{gather*}
 if and only if $X$ admits the homogeneous vertical symmetry	
\begin{gather*}
	Y= \sum_{j=1}^n\left( \sum_{i=1}^N a_{i,j}(x) \mu_{i}(y_1,\ldots,y_n)\right) \frac{\partial}{\partial y_j},
\end{gather*}
 where $\mu_{i}(y_1,\ldots,y_n)$ denotes the $i$-th monomial $($for the lexicographic order$)$ of degree~$m$
 in the~$n$ variables~$y_1,\ldots,y_n$.
 \end{Proposition}

 \begin{proof}
 Let $(\mathcal{M},\partial)$ denote the dif\/ferential module associated with $y'=Ay$.
Letting $y_1,\ldots, y_n$ denote a basis of the dual $\mathcal{M}^\star$, we see that
$X$ represents the action of $\partial$ on $\mathcal{M}^\star$
($X$ acts on the f\/irst integrals rather than on the solutions, see \cite{morales1999,MR2001, Weil1995}).
The monomials of degree $m$ in the $y_1,\ldots, y_n$ form a basis of $\operatorname{Sym}^m(\mathcal{M}^\star)$.
So, a homogeneous vertical polynomial symmetry is a~map from $\mathcal{M}^\star$ to
$\operatorname{Sym}^m(\mathcal{M}^\star)$ which furthermore commutes with $X$ (and hence with $\partial$).
Now
\begin{gather*}
\operatorname{Hom}\big(\mathcal{M}^\star, \operatorname{Sym}^m(\mathcal{M}^\star) \big)
	= (\mathcal{M}^\star)^\star \otimes \operatorname{Sym}^m(\mathcal{M}^\star)
	= \mathcal{M} \otimes \operatorname{Sym}^m(\mathcal{M}^\star),
\end{gather*}
so that
\begin{gather*}
\operatorname{Hom}_{\partial}\big(\mathcal{M}^\star, \operatorname{Sym}^m(\mathcal{M}^\star) \big)
	= \ker\big( \partial, \mathcal{M} \otimes \operatorname{Sym}^m(\mathcal{M}^\star) \big).
\end{gather*}
This shows that the coef\/f\/icients of $Y$ are exactly the rational solutions of  $y'=  {\mathcal{A}_m} y$.
 \end{proof}

\subsection{Symmetries vs Galois}\label{section6.2}

The intrinsic relation between the Galois group and the Lie algebra of symmetries
of \eqref{LE2} is made explicit by the following result.

\begin{Theorem}\label{t:Lie_Gal}
Let $Y^{(0)}$ be a polynomial vector field in $\mathbb C^n$, and $x_0$ a non-singular
point of~\eqref{LE}. There is a polynomial vertical  symmetry $Y$ of $X$ with coefficients
in~$\mathcal K$ such that $Y(x_0) = Y^{(0)}$ if and only if for each
Galoisian matrix $\sigma\in {\rm Gal}(x_0,X)$, $\sigma_*Y^{(0)} = Y^{(0)}$.
\end{Theorem}

\begin{proof}
It follows  from our def\/inition of Galois group.
By Theorem~\ref{t:sym_ci}, for each one of them there is a polynomial vertical~$Y$
symmetry such that $Y(x_0) = Y^{(0)}$ def\/ined in a neigbourhood of~$x_0$. By
Theorem~\ref{t:sym_lv} this is a solution of a Lie--Vessiot system induced
by $X$, here we consider all the homogeneous components simultaneously. Finally,
by Lemma~\ref{invariant}, this solution has coef\/f\/icients in~$\mathcal K$ if
and only if $A_*Y(x_0) = Y(x_0)$ for all Galoisian matrices at~$x_0$.
\end{proof}

Hence, the Galois group ${\rm Gal}(x_0,X)$ of (\ref{LE}) determines the Lie algebra
${\bf sym}^{<\infty}(\mathcal K)$ of polynomial vertical  $\mathcal K$-rational symmetries
in the following sense. The Lie algebra ${\bf sym}^{<\infty}(\mathcal K)$ is isomorphic
to $\mathfrak X[\mathbb C^n]^{{\rm Gal}(x_0,X)}$, the Lie algebra of polynomial vector
f\/ields f\/ixed by the action of ${\rm Gal}(x_0,X)$ by linear transformations in $\mathbb C^n$.
However, we do not have a reciprocal: in general, the Galois group is contained
in the stabilizer of the Lie algebra of $\mathcal K$-rational symmetries.

It is possible also to dualize the situation and to consider the Galois group itself as symmetries of the inf\/initesimal symmetries of
equation~(\ref{LE}) as follows. Given a polynomial vector f\/ield $Y$ in $\mathbb C^n$ a \emph{linear symmetry} of $Y$ is a non-degenerated matrix $\sigma$ such that $\sigma_* Y = Y$.
Here, $\sigma$ stands for the transformation
\begin{gather*}
y = (y_1,\ldots,y_n) \to \sigma y = \left(\sum_{j=1}^n \sigma_{1j}y_j,
\ldots, \sum_{j=1}^n \sigma_{nj}y_j\right).
\end{gather*}
If the expression in coordinates of~$Y$ is
\begin{gather*}
Y = \sum P_i(y)\frac{\partial}{\partial y_i},
\end{gather*}
then the matrix $\sigma$ is linear  symmetry of $Y$ if and only if
it satisf\/ies the equations
\begin{gather}\label{sym_pol}
P_i(\sigma y) = \sum_{j=1}^n\sigma_{ij}P_j(y).
\end{gather}
If we are looking for the linear  symmetries of a homogeneous polynomial
vector f\/ield of degree~$m$, it yields a total of $n \times {n+m-1 \choose m}$ equations.
Thus, for generic polynomial vector f\/ields of high degree the group of
linear  symmetries reduces to the identity.

\begin{Example}
Let us compute the symmetries of the quadratic vector f\/ield
\begin{gather*}Y = y_2^2\frac{\partial}{\partial y_1}\end{gather*}
in $\mathbb C^2$. Equations \eqref{sym_pol} for this particular case yield
\begin{gather*}\sigma_{21}^2y_1^2 + 2\sigma_{21}\sigma_{22}y_1y_2 + \sigma_{22}^2y_2^2 = \sigma_{11}y_2^2,
\qquad
0 = \sigma_{21}y_2^2,
\end{gather*}
equating each coef\/f\/icient, we obtain
\begin{gather*}
\sigma_{11} = \sigma_{22}^2, \qquad \sigma_{21} = 0.
\end{gather*}
Thus, the group of  linear symmetries is
\begin{gather}\label{g:quadratic}
\left\{
\left( \begin{matrix}
   \lambda^2 & \mu \\
   0 & \lambda
  \end{matrix} \right)  \colon \  \lambda\in \mathbb C^*,\   \mu \in \mathbb C
\right\}.
\end{gather}
\end{Example}

The Lie--Vessiot induced system for polynomial vertical  symmetries of arbitrary
degree~$r$ is
\begin{gather*}\frac{dY}{dx} = -[\vec v_A(x),Y].
\end{gather*}
We can consider all the homogeneous components simultaneously. Theorem~\ref{t:Lie_Gal} can be restated in the following terms:

\begin{Corollary}
Let $Y^{(0)}$ be a polynomial vector field in $\mathbb C^n$, and $x_0$ a non-singular point of equation~\eqref{LE}.
The necessary and sufficient
condition for the existence of a polynomial vertical  $\mathcal K$-rational
symmetry $Y$ of $X$ such that $Y(x_0) = Y^{(0)}$ is that the Galois group
${\rm Gal}(x_0,X)$ is contained in the group of linear symmetries
of~$Y_0$.
\end{Corollary}

\begin{Example}
Let us consider the system
\begin{gather*}
\frac{dy_1}{dx}  =  2a(x)y_1 + b(x)y_2, \qquad
\frac{dy_2}{dx}  =  a(x)y_2,
\end{gather*}
where $a(x)$, $b(x)$, are arbitrary functions in~$\mathcal K$.
A~direct computation of the Lie bracket says that $Y = y_1^2\frac{\partial}{\partial y_2}$
is a symmetry, and thus the Galois group of the equation (for any function f\/ield~$\mathcal K$)
is contained in the group~\eqref{g:quadratic}.
\end{Example}

Let us check now how the polynomial vertical  symmetries of  degrees
 zero and one look like, and what kind information about the Galois group they carry.

\subsubsection{Symmetries of degree zero}\label{section6.2.1}

The canonical isomorphism $\mathfrak X^0[\mathbb C^n] \simeq \mathbb C^n$ stated
in Section~\ref{pol_vf} tell us that the linear representation $(\Phi_0, \mathfrak X^0[\mathbb C^n])$ is just an isomorphism. In particular, if
\begin{gather*}
Y = \sum_{i=1}^n f_i(x)\frac{\partial}{\partial y_i}
\end{gather*}
then
\begin{gather*}
\Phi'(A(x))Y = \sum_{i,j=1}^n a_{ij}(x)f_j(x)\frac{\partial}{\partial y_i},
\end{gather*}
and thus we have the following result:

\begin{Proposition}
 A $n$-tuple of functions $y = (\phi_1(x),\ldots,\phi_n(x))$ is a solution of~\eqref{LE}
if and only if $Y = \sum\limits_{i=1}^n\phi_i(x)\frac{\partial}{\partial y_i}$ is a symmetry
of~\eqref{LE}.
\end{Proposition}

This proposition can be understood as an inf\/initesimal version of the superposition
principle. If $y(x)$ and $\phi(x)$ are solutions of~\eqref{LE}, then for all $\varepsilon$,
$y(x) + \varepsilon \phi(x)$ is also a solution. For a~f\/ixed~$\phi(x)$ and a~general~$y(x)$, the derivative of this monoparametric family of solutions with respect to
$\epsilon$ is a vertical vector f\/ield, namely, the symmetry~$Y$.

\begin{Remark}\label{Alambert} Given a symmetry of degree zero of~\eqref{LE} (i.e., a~solution),  we can reduce the  system  of $n$ dif\/ferential equations~\eqref{LE} to a  system of $n-1$  dif\/ferential equations, by means of a suitable
gauge transformation. This simple observation can be viewed as a generalization of the classical result of d'Alembert: the order of a linear dif\/ferential equation can be reduced by one when a particular solution is known.
\end{Remark}

\subsubsection{Linear symmetries}\label{section6.2.2}

Homogeneous polynomial vertical  symmetries of degree one are called \emph{linear symmetries}.
The homogeneous vector f\/ield $\vec h = \sum\limits_{i=1}^n y_i\frac{\partial}{\partial y_i}$
and its multiples gives us a trivial monoparametric family linear vertical
symmetries for any system of dif\/ferential equations.
The canonical isomorphism $\mathfrak X^1[\mathbb C^n] \simeq \mathfrak{gl}(n, \mathbb C)$ stated
in Section~\ref{pol_vf} tell us the that Lie--Vessiot system induced by the
representation  $(\Phi_1, \mathfrak X^1[\mathbb C^n])$ can be seen as a matrix equation.
If we write
\begin{gather*}
Y = \sum_{i,j=1}^n b_{ij} y_j\frac{\partial}{\partial y_i},
\end{gather*}
where $B = (b_{ij})$ stands for a $n\times n$ matrix of undetermined functions,
the induced system is written as
\begin{gather}\label{lax}
\frac{dB}{dx} = [A(x),B].
\end{gather}
This is the equation of isospectral deformations induced by $A(x)$ and has been exhaustively
studied. The set of rational solutions of $(\ref{lax})$ is called the \emph{eigenring}, see~\cite{Barkatou2007} for an extensive study of its properties, notably to decompose linear dif\/ferential systems.
If~$B(x)$ is a solution of~\eqref{lax} it is well known that the Jordan canonical
form of $B(x)$ does not depend on the point~$x$. Thus, given a linear symmetry $Y$
with matrix~$B(x)$, we will classify it according to its Jordan canonical form:
\begin{itemize}\itemsep=0pt
\item[(a)] If $B(x)$ has at least two dif\/ferent eigenvalues we will say that $Y$ is
a \emph{decomposer} symmetry.
\item[(b)] If all the eigenvalues of $B(x)$ are dif\/ferent we will say that
$Y$ is a \emph{complete decomposer} symmetry.
\item[(b)] If the eigenspaces of $B(x)$ are one-dimensional
we say that $Y$ is a \emph{solver} symmetry. That means that its
Jordan canonical form does not contain any block of the form
\begin{gather*}\left(  \begin{matrix}
   \lambda & 0 \\
   0 & \lambda
  \end{matrix}  \right).\end{gather*}
\end{itemize}

\looseness=-1
 The following theorem is very close to some of the results of C.~Jensen
in~\cite[Section~9]{Jensen2005}, on integration by quadratures, and results of M.A.~Barkatou in~\cite{Barkatou2007} on decomposition~-- although we obtain it  by dif\/ferent means and relate it with the linear symetries and the Galois group of the
system.

\begin{Theorem}\label{t:decomposer}
The following are equivalent:
\begin{itemize}\itemsep=0pt
\item[$(a)$] There is a decomposer symmetry $Y\in {\bf sym}(\mathcal K)$ with
$k$ different eigenvalues of multiplicity $r_1,\ldots,r_k$.
\item[$(b)$] The Galois group ${\rm Gal}(x_0,X)$ is conjugated to
a sugbroup of the group of block-diagonal matrices
${\rm GL}(r_1,\mathbb C)\times\cdots\times{\rm GL}(r_k,\mathbb C)$.
\end{itemize}
\end{Theorem}

\begin{proof}
(a)~$\Longrightarrow$~(b). Let $x_0$ be a non-singular point and $B$ the matrix def\/ined by
$Y(x_0)$, with eigenvalues $\lambda_1,\ldots,\lambda_k$ of multiplicities $r_1,\ldots,r_k$.
Let us consider the decomposition
\begin{gather}\label{Cdecomp}
\mathbb C^n = E_1\oplus \cdots \oplus E_k,
\end{gather}
where the spaces $E_i = \ker(B-\lambda_i\operatorname{Id})$ are the generalized eigenspaces of~$B$.
The group
\begin{gather*}
G = \{\sigma\in {\rm GL}(n,\mathbb C)  \colon \sigma(E_i) = E_i \ \mbox{for all} \ i=1,\ldots,k\}
\end{gather*}
is clearly conjugated to the group of block-diagonal matrices. Let us see
that all Galoisian matrices at~$x_0$ are in~$G$. If~$\sigma$ is Galoisian then,
$\sigma_*(Y(x_0)) = Y(x_0)$, but that means $\sigma B \sigma^{-1} = B$, so $\sigma$
conjugates $B$ with itself, and thus it sends generalized eigenspaces of~$B$ to themselves.

(b)~$\Longrightarrow$~(a). Let us assume that ${\rm Gal}(x_0,X)$ is conjugated to
a subgroup of the group of block-diagonal matrices. Then, we have a decomposition
of $\mathbb C^n$ in subspaces as in formula~\eqref{Cdecomp}, such that for all
$\sigma\in {\rm Gal}(x_0,X)$, and index $i=1,\ldots,k$, $\sigma(E_i) = E_i$. Let
us consider the following linear vector f\/ields $\vec h_i$ in $\mathbb C^n$
for $i=1,\ldots,k$ def\/ined by properties:
\begin{gather*}
\vec h_i|_{E_i} = \vec h,
\quad\mbox{where} \ \vec h\mbox{ stands for the Euler homogeneous vector f\/ield},\\
\vec h_i|_{E_j} = 0, \quad \mbox{if} \ i\neq j.
\end{gather*}
Let us consider $\mu_1,\ldots,\mu_k$ dif\/ferent complex numbers and def\/ine
\begin{gather*}
Y^{(0)} = \sum_{i=1}^k \mu_i\vec h_i.
\end{gather*}
 $Y^{(0)}$ is stabilized by any Galoisian matrix and then, by
Lemma \ref{invariant}, there is $\mathcal K$-rational symmetry~$Y$ whose value
at~$x_0$ is~$Y^{(0)}$. This symmet\-ry~$Y$ is the decomposer symmetry of the
statement.
\end{proof}

{\samepage
\begin{Corollary}\label{c:complete}
The following are equivalent:
\begin{itemize}\itemsep=0pt
\item[$(a)$] There is a complete-decomposer linear symmetry in ${\bf sym}^1_X(\mathcal K)$.
\item[$(b)$] The Galois group ${\rm Gal}(x_0,X)$ is conjugated to subgroup of
the group of diagonal matrices $(\mathbb C^*)^n\subset {\rm GL}(n,\mathbb C)$.
\end{itemize}
\end{Corollary}

}

\begin{proof} The statement is the particular case
of Theorem~\ref{t:decomposer} in which all the eigenvalues are simple.
\end{proof}

\begin{Remark}\label{abe-sym} The existence of a complete-decomposer linear symmetry
implies the existence of a $n$-dimensional abelian Lie algebra
of symmetries. Let us consider a complete-decomposer linear symmetry
$Y$ and $x_0$ a non-singular point. Let $B$ be the matrix of $Y(x_0)$,
and $v_1,\ldots,v_n$ be a basis of eigenvector of $B$. As before, we def\/ine
vector f\/ields:
\begin{gather*}
\vec h_i(v_i) = \vec h(v_i),
\quad\mbox{where} \ \vec h \ \mbox{stands for the homogeneous vector f\/ield},\\
\vec h_i(v_j) = 0, \quad\mbox{if} \ i\neq j.
\end{gather*}
It is easy to check that the matrices of the vector f\/ields $\vec h_i$ have common
eigenvectors and then $[\vec h_i,\vec h_j]= 0$.
For all Galoisian matrix $\sigma$ at $x_0$ we have $\sigma_*(\vec h_i) = \vec h_i$, and
thus by Lemma~\ref{invariant} there are linear vertical  $\mathcal K$-rational symmetries
$\vec H_1,\ldots,\vec H_n$ such that $\vec H_i(x_0) = \vec h_i$. They form
a~$n$-dimensional abelian Lie algebra.
\end{Remark}

\begin{Theorem}\label{t:solver}
If there is a solver symmetry $Y\in {\bf sym}(\mathcal K)$ then the Galois group
${\rm Gal}(x_0,X)$ is conjugated to a subgroup of the triangular group.
\end{Theorem}

\begin{proof}
Let us consider f\/irst  the case in which $B$, the matrix of $Y(x_0)$, has only one
eigen\-va\-lue~$\lambda$ of multiplicity~$n$. For each Galoisian matrix~$\sigma$ at~$x_0$ we have $\sigma B\sigma^{-1} = B$.
We have a chain of subspaces
\begin{gather*}0 \subset \ker(B - \lambda \operatorname{Id}) \subset \ker(B - \lambda \operatorname{Id})^2 \subset
\cdots \subset \ker(B - \lambda \operatorname{Id})^{n-1} \subset {\mathbb C}^n.
\end{gather*}
In general, $\sigma\left(\ker(B - \lambda \operatorname{Id})^{j}  \right) =
\ker(\sigma B \sigma^{-1} - \lambda \operatorname{Id})^{j}$
and thus Galoisian matrices respect the chain of subspaces. In other words, they
are triangular matrices in some suitable basis. For the general case, with
dif\/ferent eigenvalues, we f\/irst consider the decomposition of the group
by block-diagonal matrices given in Theorem~\ref{t:decomposer}, and then
we apply the above argument.
\end{proof}

The results of this subsection supported Lie's idea that Lie symmetries are useful for
the integrability of the dif\/ferential equation~(\ref{LE}) by quadratures or at least its partial integrability or reduction:
 \begin{itemize}\itemsep=0pt
 \item [(1)] By Remark~\ref{Alambert}, the existence of a symmetry of order zero implies a reduction of the order.
 \item[(2)] The decomposition in block-diagonal form of the Galois group implies that, by means of
 a gauge transformation, we can transform  the equation~(\ref{LE}) in a direct sum of linear dif\/ferential equations (Kolchin or Lie--Kolchin reduction, see~\cite{BM2012,PutSinger2003}). By Theorem~\ref{t:decomposer}, we fall in this case for a decomposer symmetry.
 \item[(3)] By Remark~\ref{abe-sym} and Theorem~\ref{t:solver}, the existence of either a~complete-decomposer or a~solver symmetry implies the solvability of the equation by Liouvillian functions.
 \end{itemize}

 Also it is not dif\/f\/icult  to obtain some results for the Hamiltonian symmetries in the  symplectic case, i.e., for non-autonomous linear Hamiltonian systems. In some sense, this approach would shed light on the references~\cite{morales1999, MR2001} from the Lie point of view.

\appendix

\section{Symmetries of higher-order equations vs f\/irst-order systems}\label{Appendix_A}

  There are two dif\/ferent ways to present the theory of linear dif\/ferential equations.
The f\/irst one deals with a single higher-order linear dif\/ferential equation:
\begin{gather}\label{e_operator}
\frac{d^n y}{dx^n} + a_{n-1}(x)\frac{d^{n-1}y}{dx^{n-1}} + \dots + a_0(x)y = 0.
\end{gather}
The second one deals with a system of f\/irst-order linear dif\/ferential equations
\begin{gather}\label{c_system}
\frac{d}{dx} \left(
\begin{matrix} y_0 \\ y_1 \\ \vdots \\ y_{n-1} \end{matrix}
\right) = \left(
\begin{matrix}
0      & 1 & 0      & \ldots & 0 \\
0      & 0 & 1      & \ldots & 0 \\
\vdots & \vdots   & \ddots &  \ddots     & \vdots \\
- a_0(x)  & - a_1(x) & \ldots & \ldots & -a_{n-1}(x)
\end{matrix}
\right)
\left(
\begin{matrix} y_0 \\ y_1 \\ \vdots \\ y_{n-1} \end{matrix}
\right).
\end{gather}

The system \eqref{c_system} is called the companion system of~\eqref{e_operator}.
The variable~$y_i$ represents the $i$-th derivative of the function~$y$.
In this paper, we have considered systems instead of higher-order equations.
There is a pragmatic reason: although the dif\/ferential Galois theories for equations~\eqref{e_operator} and~\eqref{c_system} are the same (see, for instance \cite[Section~2.1]{PutSinger2003}),
their Lie sym\-met\-ry theories are not.
Higher-order equations have less symmetries than systems of f\/irst-order dif\/ferential equations.

By def\/inition, an \emph{infinitesimal point symmetry} of
\eqref{e_operator} is a vector f\/ield
\begin{gather*}
Y = \xi(x,y)\frac{\partial}{\partial x} + \eta(x,y)\frac{\partial}{\partial y}
\end{gather*}
in the plane~$x$,~$y$ whose f\/low maps solutions to solutions.
Let us consider the vector f\/ield
\begin{gather*}
X = \frac{\partial}{\partial x} + y_1\frac{\partial}{\partial y_0} + \dots - (a_0(x)y_0 + \dots + a_{n-1}(x)y_{n-1})
\frac{\partial}{\partial y_{n-1}}.
\end{gather*}
The vector f\/ield $Y$ extends to a
unique vector f\/ield $\tilde Y$ in the jet space of coordinates $x,y_0$, $\ldots,y_{n-1}$
\begin{gather*}
\tilde Y = \xi(x,y)\frac{\partial}{\partial x} + \eta(x,y_0)\frac{\partial}{\partial y_0} +
\eta_1(x,y_0,y_1)\frac{\partial}{\partial} + \cdots +\eta_{n-1}(x,y_0,\ldots,y_{n-1})\frac{\partial}{\partial y_{n-1}},
\end{gather*}
satisfying the conditions (see \cite[Section~2.3]{Olver})
\begin{gather*}
[\tilde Y,X] \in (X),\\
 {\rm Lie}_{\tilde Y}(dy_0-y_1dx, \ldots, dy_{n-2}-y_{n-1}dx) \subseteq (dy_0-y_1dx, \ldots, dy_{n-2}-y_{n-1}dx).
\end{gather*}
 It is known that the Lie algebra of point symmetries (in some open subset) of a linear dif\/ferential equation of order $\geq 2$ is f\/inite-dimensional. The above conditions allow us to generalize the idea of inf\/initesimal point symmetry.
 Any vector f\/ield in the jet space of coordinates
$x,y_0,\ldots,y_{n-1}$ is called an \emph{infinitesimal contact symmetry} if is satisf\/ies
\begin{gather*}
[Z,X] \in (X),\\
 {\rm Lie}_{Z}(dy_0-y_1dx, \ldots, dy_{n-2}-y_{n-1}dx) \subseteq (dy_0-y_1dx, \ldots, dy_{n-2}-y_{n-1}dx).
\end{gather*}
Inf\/initesimal point symmetries form a Lie subalgebra of the Lie algebra of inf\/initesimal contact symmetries.

On the other hand, an \emph{infinitesimal symmetry} of the system~\eqref{c_system} is a vector f\/ield~$Z$ such that $[Z,X] \in (X)$.
See for instance~\cite{Athorne1997} and \cite[pp.~12--16]{KLR}. It is clear that the Lie algebra of symmetries of the companion system~\eqref{c_system} contains the Lie algebra of inf\/initesimal contact symmetries of the higher-order dif\/ferential equation~\eqref{e_operator}.

In this paper, we have explored the relation between some Lie algebras of symmetries of a~f\/irst-order system~\eqref{c_system} and
its dif\/ferential Galois group. The relation between the Lie algebra of inf\/initesimal point symmetries of an operator~\eqref{e_operator} and its dif\/ferential Galois group has been investigated, with rather negative results, in~\cite{Athorne1997} and~\cite{Vanderput2002}.

\subsection*{Acknowledgements}

The authors thank their colleagues for interesting discussions that encouraged
them to write this paper, especially those attending the meeting Algebraic
Methods in Dynamical Systems~2014. David Bl\'azquez-Sanz acknowledges Universidad Nacional de Colombia for supporting his research through project  ref.~HERMES-27984. Juan J.~Morales-Ruiz research  has been also partially supported by the Spanish MINECO-FEDER Grant MTM2012-31714. We thank the anonymous referees for their remarks and suggestions.

%\cite{CasaleWeil2015,FultonHarris}
\pdfbookmark[1]{References}{ref}
\LastPageEnding

\end{document}